\documentclass[envcountsect]{svmult}

\usepackage{tikz}
\usetikzlibrary{arrows.meta,arrows}

\def\N{\mathbb{N}}

\def\R{\mathbb{R}}

\def\Aut{\operatorname{Aut}}
\def\spec{\operatorname{spec}}
\def\Fix{\operatorname{Fix}}

\def\deg{\operatorname{deg}}
\def\sp{\operatorname{span}}
\def\dist{\operatorname{dist}}
\definecolor{pink}{RGB}{250, 191, 183}
\definecolor{yellow}{RGB}{253, 249, 196}
\definecolor{blue}{RGB}{178, 226, 242}
\definecolor{orange}{RGB}{255, 218, 158}
\definecolor{purple}{RGB}{197, 198, 200}

\usepackage{type1cm}        % activate if the above 3 fonts are
\usepackage{makeidx}         % allows index generation
\usepackage{graphicx}        % standard LaTeX graphics tool
\usepackage{multicol}        % used for the two-column index
\usepackage[bottom]{footmisc}% places footnotes at page bottom

\usepackage{newtxtext}       % 
\usepackage{newtxmath}       % selects Times Roman as basic font

\usepackage{comment}

\makeindex             % used for the subject index
\begin{document}

\title*{A Symmetry-Based Classification of Synchrony in Tree Networks}
% Use \titlerunning{Short Title} for an abbreviated version of
% your contribution title if the original one is too long
\author{Nícolas Brito and Miriam Manoel}
% Use \authorrunning{Short Title} for an abbreviated version of
% your contribution title if the original one is too long
\institute{Nícolas Brito \at Instituto de Ciências Matemáticas e de Computação, Universidade de São Paulo (ICMC-USP), 
Av. Trabalhador São-carlense, 400 - Centro, São Carlos, SP, Brasil, 13566-590. \email{nicolas.r.brito@usp.br}
\and Miriam Manoel \at Instituto de Ciências Matemáticas e de Computação, Universidade de São Paulo (ICMC-USP), 
Av. Trabalhador São-carlense, 400 - Centro, São Carlos, SP, Brasil, 13566-590. \email{miriam@icmc.usp.br} }
%
% Use the package "url.sty" to avoid
% problems with special characters
% used in your e-mail or web address
%
\maketitle

\begin{abstract}
{Coupled cell systems model interacting dynamical units and provide a natural framework for studying synchrony phenomena arising from collective behavior. Graph symmetries often induce such patterns, but certain networks exhibit additional synchronies not associated with automorphisms, commonly referred to as exotic synchronies. In undirected asymmetric graphs, any synchrony, if present, must be non-symmetry-induced, and determining when such exotic patterns occur remains a challenging structural problem. In this work, we address this question for networks whose underlying coupling graph is a tree, a class of graphs that naturally models hierarchical interactions among elements. We prove that exotic synchronizations do not arise in tree-type networks, showing that every balanced coloring is a fixed-point coloration determined by graph automorphisms. Furthermore, we identify the importance of the role played by the leaves of a graph in this context. Beyond existence results, we investigate the dynamical consequences of these structures by analyzing the linear stability of equilibria and the Lyapunov stability of synchrony subspaces for admissible vector fields defined on tree networks. Particular attention is devoted to cherry-type configurations, where local symmetries generated  by leaves attached to a common vertex influence the stability properties of the associated synchronous states, thereby clarifying how the combinatorial architecture of trees constrains both the emergence and the stability of synchrony.}
\end{abstract}



\section{Introduction}
\label{sec:1}

Networks of interacting dynamical systems arise in a wide range of scientific contexts \cite{ADR, Alessio, VR, HR}. A central phenomenon in such systems is synchrony, where subsets of units evolve identically in time despite the complexity of the underlying interactions. The theory of coupled cell systems, developed by Golubitsky, Stewart, and collaborators \cite{DS, GNS, GSP, GST}, provides a mathematical framework for studying these phenomena, relating dynamical properties to the structure of an associated coupling graph. On one hand, synchronization is often associated with an invariant subspace induced by symmetries of the underlying network. These are the patterns that arise as fixed-point subspaces of a subgroup of the automorphism group of the graph. Some networks, on the other hand, admit additional invariant subspaces not induced by symmetry, the so-called exotic synchrony \cite{AS}. 

For asymmetric networks, a synchrony, if present, is exotic. Explicit examples of such behavior are known, but a general characterization of when exotic synchrony can or cannot occur remains an open and challenging problem, even for networks with one type of cell and one type of arrow. In the directed case, the existence of non-trivial synchronies is common and easily observable even for networks with two or tree vertices, as shown in Fig.~\ref{fig:exotic}.
However, in the undirected case, when we restrict attention to simple networks -- those modeled by undirected graphs with no multiple edges and no self-loops -- our investigations suggest that the existence of synchrony patterns in asymmetric networks is extremely rare. In fact, recent numerical simulations performed on two hundred random asymmetric undirected graphs with 10 to 25 vertices and with probability of having an edge between them varying from 0.2 to 0.6, showed that none of these graphs admit a non-trivial balanced coloration.

Regularity of a graph (all vertices with same valency) is a sufficient condition for total synchrony, namely, all network cells evolving identically. Some simple asymmetric graphs are regular, so emergence of synchrony in such structures is possible. The total synchrony, however, is not necessarily the only possible pattern that may arise in asymmetric networks. In fact, the Frucht graph -- a 3-regular asymmetric graph with 12 vertices -- not only admits the total synchrony but also supports 2-cluster and 3-cluster synchrony patterns (Fig.~\ref{fig:Frucht}).

Asymmetric graphs are particularly relevant from a structural viewpoint. A classical result of Erdős and Rényi \cite{ER} shows that the probability of a random graph being asymmetric tends to $1$ as the number of vertices goes to infinity. Consequently, symmetry should be regarded as a non-generic feature in large undirected networks, and synchronization mechanisms that do not rely on automorphisms become essential for understanding typical network behavior. This observation further motivates the investigation of exotic synchrony in asymmetric graphs as a central problem in the theory of coupled cell systems. 

In this work, we focus on networks whose underlying coupling graph is a tree. From a theoretical perspective, trees represent a minimal class for the study of synchrony in asymmetric graphs, making them particularly suitable to investigate the limits of exotic synchrony. Since every connected graph contains at least one spanning tree that captures its connectivity structure, understanding synchronization mechanisms on trees provides insight into more general network topologies. We point out that there exist infinitely many asymmetric trees \cite{ER}, the smallest having order 7 (Fig.~\ref{fig:assymetric_tree}), which reinforce the relevance of this class as a nontrivial and structurally rich family.

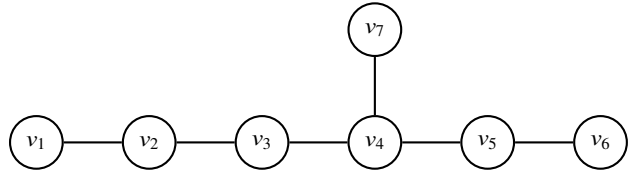
\begin{figure}[htbp]
\label{fig:assymetric_tree}
    \centering
        \begin{tikzpicture}[node distance={1.5cm}, thick, main/.style = {draw, circle, minimum size=0.7cm}] 
            \node[main] (1) {$v_1$}; 
            \node[main] (2) [right of=1] {$v_2$};
            \node[main] (3) [right of=2] {$v_3$}; 
            \node[main] (4) [right of=3] {$v_4$};
            \node[main] (5) [right of=4] {$v_5$};
            \node[main] (6) [right of=5] {$v_6$};
            \node[main] (7) [above of=4] {$v_7$};
            \draw (1) -- (2);
            \draw (2) -- (3);
            \draw (3) -- (4);
            \draw (4) -- (5);
            \draw (4) -- (7);
            \draw (5) -- (6);
        \end{tikzpicture}
    \caption{Example of an asymmetric tree of order $7$.}
\end{figure}

From a dynamical systems perspective, tree-like structures arise as models of a wide range of systems and encode hierarchical interaction structures. Such structures appear, for instance, in models of neuronal dendritic \cite{Dentric}, vascular and respiratory networks \cite{WBE} and phylogenetic relationships \cite{Filo}. In these settings, the acyclic property reflects interaction pathways, such as undirected flow, resource distribution, or signal transmission. 

Vertices of degree one, often referred to as \textit{pendant} vertices, play a key role in these systems, often representing boundary components whose limited connectivity can significantly influence global dynamical behavior. These vertices are also known as {\it leaves} in tree-type graphs and are fundamental in the characterization of synchronization (Propositions~\ref{prop:distleaf}, \ref{prop:samecolorleaf} and Lemma~\ref{lem:classcolor}), as shown in Section~\ref{sec:3}. Consequently, we obtain the main result of this section: tree-type networks do not admit exotic synchrony patterns (Theorem~\ref{thm:no_exotic_trees}). This fact places trees in the class of networks for which synchrony is completely determined by symmetries, establishing a first step for understanding the relationship between synchrony and asymmetry in undirected graphs. 

Cherry configurations -- groups of leaves attached to a common neighbor -- play a significant structural role in trees, and their importance is reflected in their ubiquity: the probability that a random tree contains at least one cherry tends to one as the order of the graph increases \cite{ER}. A graph containing a cherry has a symmetry, namely the permutation that interchanges the  attached leaves. For this reason, beyond existence of a synchrony, we also investigate its stability induced by cherries (Proposition~\ref{prop:alpha} and Corollary~\ref{cor:m-cherry}). We shall see that their local arrangement influences the Lyapunov stability of the synchrony subspace generated by the cherries, as presented in Theorem~\ref{thm:Lyapunov} and Corollary~\ref{cor:final}. 

The paper is organized as follows. Section~\ref{sec:2} is devoted to preliminaries, where we review basic definitions and results on coupled cell systems and synchronization. In Section~\ref{sec:3} we characterize synchrony patterns in tree-type networks through leaves and the pruning process. Finally, in Section~\ref{sec:4}, we analyze the linear stability and Lyapunov stability in admissible tree networks.

%%%%%%%%%%%%%%%% SECTION 2 %%%%%%%%%%%%%%%%%%%%%%%%%%%%%%%%%%

\section{Coupled Cell Systems}
\label{sec:2}

This section is devoted to preliminary background on coupled cell systems and the role of symmetries in the emergence of synchrony patterns. Throughout the paper, we shall use $G$ indistinctly to denote either a network or the associated graph. Network terminology is standard in the coupled cell systems literature, where it is commonly employed when the graph is viewed from a dynamical systems perspective. Following \cite{GST}, a \textit{coupled cell network} $G$ consists of a finite set $V=\{v_1,\cdots,v_n\}$ of \textit{cells} (or vertices), also denoted $V(G)$ when necessary, and a finite set $E$ of \textit{edges} that describe coupling between cells. In each set, we introduce a equivalence relation that describe the types of cells and edges.

In order to relate the dynamics of cells that have the same couplings, we define the \textit{input set} of $c\in V$ as $I(c) = \{e\in E:\mathcal{H}(e) = c\}$ and  say that two cells $c,d$ are \textit{input isomorphic} if there is an isomorphism between their input sets that preserves edge type. In this setting, the network is said to be \textit{homogeneous} if all cells are input isomorphic.

For each cell $c\in V$ the \textit{cell phase space} $P_c$ is defined as a nonzero finite-dimensional vector space, with $P_c=P_d$ if cells $c$ and $d$ are of the same type. Moreover, $P =\prod_{c\in V} P_c$ is the \textit{total phase space} of the network, and 
$$P_{\mathcal{T}(I(c))} = P_{\mathcal{T}(e_1)} \times P_{\mathcal{T}(e_2)} \times \cdots \times P_{\mathcal{T}(e_{|I(c)|})}$$
the \textit{coupling phase space}, with $e_k$ ranging over $I(c)$, where $\mathcal{T}: E \to C$ is the \textit{tail} function which returns the source cell of the edge.

The study of the evolution of the coupled system over time is done through an autonomous ODE $\dot x = f(x)$, and the vector field $f:P\to P$ must be compatible with the coupling dynamics and respect the interactions described by the associated graph. This is the category of \textit{admissible} vector fields: for all $c\in V$, the component $f_c(x)$ of the vector field depends only on the variables $x_c$ and the variables $x_{\mathcal{T}(I(c))}$, that is, there exists $\tilde{f}_c: P_c\times P_{\mathcal{T}(I(c))} \to P_c$ such that $f_c(x) = \tilde{f}_c\left(x_c,x_{\mathcal{T}(I(c))}\right )$. Also, input isomorphic cells share the same governing function, modulo a relabelling of their coordinates through the input isomorphism. For an example of an admissible vector field on a tree network, see Example~\ref{ex:nonlinear}.\\

In this formalism, a synchrony is an equivalence relation $\bowtie$ on $V$ such that $c\bowtie d$ if and only if $x_c(t) = x_d(t)$ for all $t$. Given $P$ the total phase space of the network and $\bowtie$ an equivalence relation that refines $\sim_V$, we define the \textit{polydiagonal subspace} as  
$$\Delta_{\bowtie} = \{x\in P : x_c=x_d \text{ whenever } c\bowtie d, \forall c,d\in V\}.$$
We say that $\bowtie$ is \textit{robustly polysynchronous} if $\Delta_{\bowtie}$ is invariant under any admissible vector field.

On the other hand, an equivalence relation $\bowtie$ on $V$ is said to be \textit{balanced} if for all $c\bowtie d$, $c,d\in V$, there exists an input isomorphism $\beta:I(c)\to I(d)$ such that $\mathcal{T}(e)\bowtie\mathcal{T}(\beta(e))$ for all $e\in I(c)$. Taking a distinct color for each $\bowtie$-class, the coloring is called balanced if $\bowtie$ is a balanced relation.

A central result in the literature provides a complete characterization of synchrony patterns: an equivalence relation $\bowtie$ on $V$ is robustly polysynchronous if and only if $\bowtie$ is balanced. As a balanced relation $\bowtie$ refines $\sim_I$, the study of the network dynamics through \textit{synchrony clusters} gives rise to the concept of a \textit{quotient network}, denoted by $G_{\bowtie}$. Its vertex set is given by the quotient $V_{\bowtie} = V/\bowtie$ and its edge set $E_{\bowtie}$ by the projections of the edges of the original network onto the $\bowtie$ equivalence classes. 
Balanced relations exhibit distinctive features in undirected trees. Recall that trees are connected graphs without cycles. In this setting, balanced relations arise in a particularly structured way, a fact that motivated the present study. An illustration of this phenomenon is given in Fig.~\ref{fig:tree_balanced}, which shows a balanced 3-coloring of the undirected tree with ten vertices and its quotient networks.

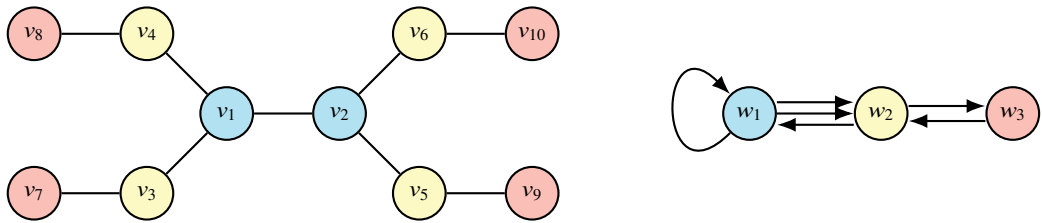
\begin{figure}[htbp]
\label{fig:tree_balanced}
    \centering
    \begin{tikzpicture}
        \node (1) at (0,0) {\begin{tikzpicture}[node distance={1.5cm}, thick, main/.style = {draw, circle, minimum size=0.7cm}] 
            \node[main, fill=blue] (1) {$v_1$}; 
            \node[main, fill=blue] (2) [right of=1] {$v_2$};
            \node[main, fill=yellow] (3) [below left of=1] {$v_3$}; 
            \node[main, fill=yellow] (4) [above left of=1] {$v_4$};
            \node[main, fill=yellow] (5) [below right of=2] {$v_5$};
            \node[main, fill=yellow] (6) [above right of=2] {$v_6$};
            \node[main, fill=pink] (7) [left of=3] {$v_7$};
            \node[main, fill=pink] (8) [left of=4] {$v_8$};
            \node[main, fill=pink] (9) [right of=5] {$v_9$};
            \node[main, fill=pink] (10) [right of=6] {$v_{10}$};
            
            \draw (1) -- (2);
            \draw (1) -- (3);
            \draw (1) -- (4);
            \draw (2) -- (5);
            \draw (2) -- (6);
            \draw (3) -- (7);
            \draw (4) -- (8);
            \draw (5) -- (9);
            \draw (6) -- (10);
        \end{tikzpicture}};

        \node (2) at (7.5,0) {\begin{tikzpicture}[node distance={1.75cm}, thick, main/.style = {draw, circle, minimum size=0.7cm}] 
            \node[main, fill=blue] (1) {$w_1$}; 
            \node[main, fill=yellow] (2) [right of=1] {$w_2$};
            \node[main, fill=pink] (3) [right of=2] {$w_3$}; 

            \draw[-{Latex[length=2.5mm]}]  (1) to[out=-135,in=135, looseness=7] (1);
                        
            \draw [-{Latex[length=2.5mm]}] (1) -- (2);
            \draw [-{Latex[length=2.5mm]}] ([yshift=0.15 cm]1.east) -- ([yshift=0.15 cm]2.west);
            \draw [-{Latex[length=2.5mm]}] ([yshift=-0.15 cm]2.west) -- ([yshift=-0.15 cm]1.east);

            \draw[-{Latex[length=2.5mm]}]  ([yshift=0.1 cm]2.east) -- ([yshift=0.1cm]3.west);
            \draw [-{Latex[length=2.5mm]}] ([yshift=-0.1 cm]3.west) -- ([yshift=-0.1 cm]2.east);
            
        \end{tikzpicture}};
    \end{tikzpicture}
     
    \caption{Example of a balanced coloring on a 10-cell undirected tree (left) and its quotient network (right).}
\end{figure}
    
Symmetries encode structural invariance through vertex permutation, which translates into invariant subspaces and synchrony patterns when viewed dynamically. Following \cite{AS}, a symmetry in a coupled cell network $G$ is a pair of bijections $\sigma = (\sigma_V,\sigma_E)$ on $V$ and $E$, respectively, that preserves cell type, edge type, and adjacency. These are the elements of the \textit{automorphism group} ${\rm Aut}(G)$. We call the network \textit{asymmetric} if the group is trivial. If $\Gamma \subset {\rm{Aut}}(G)$ is a subgroup, then the relation $\bowtie_{\Gamma}$ on $V$ given by
$$c\bowtie_{\Gamma} d \iff \text{ there exists } \sigma\in\Gamma \text{ such that } \sigma_{V}(c) = d$$
is a balanced equivalence relation. Moreover, it is straightforward to see that $\Delta_{\bowtie_{\Gamma}}$ is equal to the \textit{fixed-point subspace} $\operatorname{Fix}(\Gamma)= \{x\in P: \gamma x = x, \forall \gamma\in\Gamma\}$. Therefore, we call a balanced relation $\bowtie$ on a coupled cell network $G$ as a \textit{fixed-point coloration} if $\Delta_{\bowtie} = \operatorname{Fix}({\Gamma})$ for some $\Gamma\subset{\rm Aut}(G)$ subgroup; otherwise, it is called \textit{exotic}.

\begin{figure}[htbp]
    \centering
    \begin{tikzpicture}
       \node (1) at (0,0) {\begin{tikzpicture}[node distance={1.75cm}, thick, main/.style = {draw, circle, minimum size=0.7cm}] 
            \node[main, fill=blue] (1) {$v_1$}; 
            \node[main, fill=blue] (2)  [right of=1] {$v_2$};
        
            \draw[-{Latex[length=2.5mm]}]  (1) to[out=-135,in=135, looseness=7] (1);
            
            \draw[-{Latex[length=2.5mm]}]  (1) -- (2);
            \end{tikzpicture}};

        \node (2) at (8,0) {\begin{tikzpicture}[node distance={1.75cm}, thick, main/.style = {draw, circle, minimum size=0.7cm}] 
            \node[main, fill=blue] (1) {$v_1$}; 
            \node[main, fill=pink] (2)  [below left of=1] {$v_2$};
            \node[main, fill=pink] (3)  [below right of=1]{$v_3$};

            \draw[-{Latex[length=2.5mm]}]  (1) -- (2);
            
            \draw[-{Latex[length=2.5mm]}] ([xshift=0.05 cm]2.north) -- ([yshift=-0.05 cm]1.west);
            
            \draw[-{Latex[length=2.5mm]}]  (1) -- (3);
        
            \draw[-{Latex[length=2.5mm]}]  ([yshift=0.1 cm]2.east) -- ([yshift=0.1cm]3.west);
            \draw [-{Latex[length=2.5mm]}] ([yshift=-0.1 cm]3.west) -- ([yshift=-0.1 cm]2.east);
            \end{tikzpicture}}; 
    \end{tikzpicture}   
    \caption{Exotic patterns of synchrony on a 2-vertex network (left) and on a 3-vertex network (right).}
    \label{fig:exotic}
\end{figure}
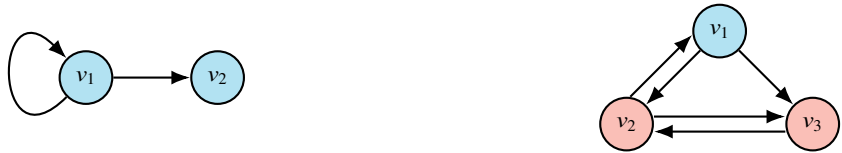

A graph with trivial automorphism group provides a sufficient condition for a balanced relation to be exotic. In Fig.~\ref{fig:exotic} we present the simplest examples of exotic colorings in directed networks with two and tree cells. In Fig.~\ref{fig:Frucht} we also illustrate two examples of exotic balanced relations in the Frucht graph, which is a classical asymmetric undirected graph.

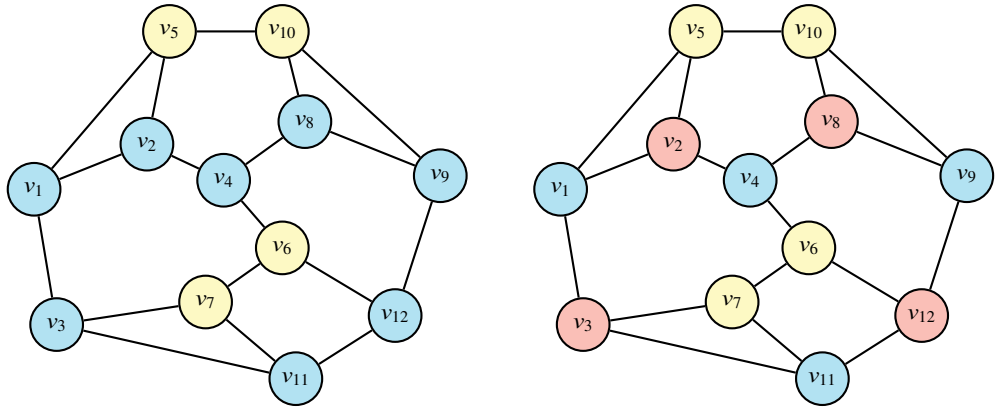
\begin{figure}[htbp]
\centering
    \begin{tikzpicture}
        \node (1) at (0,0) {\begin{tikzpicture}[node distance={1.5cm}, thick, main/.style = {draw, circle, minimum size=0.7cm}, scale=0.6]
        
            \node[main, fill=blue] (1) at (0, 0) {$v_1$};
            \node[main, fill=blue] (2) at (2.5, 1) {$v_2$};
            \node[main, fill=blue] (3) at (0.5, -3) {$v_3$};
            \node[main, fill=blue] (4) at (4.2, 0.2) {$v_4$};
            \node[main, fill=yellow] (5) at (3, 3.5) {$v_5$};
            \node[main, fill=yellow] (6) at (5.5, -1.3) {$v_6$};
            \node[main, fill=yellow] (7) at (3.8, -2.5) {$v_7$};
            \node[main, fill=blue] (8) at (6, 1.5) {$v_8$};
            \node[main, fill=blue] (9) at (9, 0.3) {$v_9$};
            \node[main, fill=yellow] (10) at (5.5, 3.5) {$v_{10}$};
            \node[main, fill=blue] (11) at (5.8, -4.2) {$v_{11}$};
            \node[main, fill=blue] (12) at (8, -2.8) {$v_{12}$};
        
            \draw (1) -- (2);
            \draw (1) -- (3);
            \draw (1) -- (5);
            \draw (2) -- (4);
            \draw (2) -- (5);
            \draw (3) -- (7);
            \draw (3) -- (11);
            \draw (4) -- (6);
            \draw (4) -- (8);
            \draw (5) -- (10);
            \draw (6) -- (7);
            \draw (6) -- (12);
            \draw (7) -- (11);
            \draw (8) -- (9);
            \draw (8) -- (10);
            \draw (9) -- (10);
            \draw (9) -- (12);
            \draw (11) -- (12);
            
            \end{tikzpicture}};

        \node (2) at (7,0) {\begin{tikzpicture}[node distance={1.5cm}, thick, main/.style = {draw, circle, minimum size=0.7cm}, scale=0.6]
        
            \node[main, fill=blue] (1) at (0, 0) {$v_1$};
            \node[main, fill=pink] (2) at (2.5, 1) {$v_2$};
            \node[main, fill=pink] (3) at (0.5, -3) {$v_3$};
            \node[main, fill=blue] (4) at (4.2, 0.2) {$v_4$};
            \node[main, fill=yellow] (5) at (3, 3.5) {$v_5$};
            \node[main, fill=yellow] (6) at (5.5, -1.3) {$v_6$};
            \node[main, fill=yellow] (7) at (3.8, -2.5) {$v_7$};
            \node[main, fill=pink] (8) at (6, 1.5) {$v_8$};
            \node[main, fill=blue] (9) at (9, 0.3) {$v_9$};
            \node[main, fill=yellow] (10) at (5.5, 3.5) {$v_{10}$};
            \node[main, fill=blue] (11) at (5.8, -4.2) {$v_{11}$};
            \node[main, fill=pink] (12) at (8, -2.8) {$v_{12}$};
    
            \draw (1) -- (2);
            \draw (1) -- (3);
            \draw (1) -- (5);
            \draw (2) -- (4);
            \draw (2) -- (5);
            \draw (3) -- (7);
            \draw (3) -- (11);
            \draw (4) -- (6);
            \draw (4) -- (8);
            \draw (5) -- (10);
            \draw (6) -- (7);
            \draw (6) -- (12);
            \draw (7) -- (11);
            \draw (8) -- (9);
            \draw (8) -- (10);
            \draw (9) -- (10);
            \draw (9) -- (12);
            \draw (11) -- (12);
            
        \end{tikzpicture} };
    \end{tikzpicture}
    \caption{2-balanced coloring (left) and 3-balanced coloring (right) in the Frucht graph.}
    \label{fig:Frucht}
\end{figure}

%%%%%%%%%%%%%%%% SECTION 3 %%%%%%%%%%%%%%%%%%%%%%%%%%%%%%%%%%

\section{Synchrony Patterns in Tree-type Networks}
\label{sec:3}

From now on, the term \textit{graph} always refers to a simple graph, that is, undirected with one type of vertex, one type of edge and without multiple edges or loops. It is well known that a homogeneous network modeled by a cycle graph does not admit exotic balanced relations \cite{AM2024}. Cycle graphs are the first class of networks whose synchrony patterns are completely determined by their automorphism group. In this section, we show that the same holds for trees, namely, that no exotic synchrony occurs in this class of graphs.

Two distinct vertices $c,d$ of a graph $G$ are \textit{adjacent}, $c\sim d$, if there is an edge between $c$ and $d$. If $G$ has order $n$ with vertex set $V = \{v_1,\cdots,v_n\}$, the \textit{adjacency matrix} of $G$ is the $n\times n$ matrix $A=(a_{ij})$,
$$a_{ij} = \begin{cases}
      1, & \text{if } v_i \sim v_j\\
      0, & \text{otherwise.}
  \end{cases}$$
If $G$ is simple, then $A$ is symmetric with zero diagonal. The \textit{degree} $\deg(c)$ of a vertex $c$ is the number of vertices adjacent to it, which is also the cardinality of the input set of $c$. Finally, recall that a \textit{path} on $G$ is a sequence of distinct vertices $p = (v_1,\cdots,v_k)$ such that $v_i\sim v_{i+1}$ for all $i\in\{1,\cdots,k-1\}$. Its \textit{length} is defined as $k-1$, and the distance between $v_1$ and $v_k$, denoted by $\dist(v_1,v_k)$, is the length of a shortest path connecting them.

\subsection{Properties of Balanced Relations}

The notion of a multiset will be useful in this subsection. A \textit{multiset} $B$ is a collection of elements in which repetitions are allowed and recorded with multiplicity. Given two multisets $B_1$ and $B_2$, their sum $B_1+B_2$ is defined by combining all elements of $B_1$ and $B_2$, adding multiplicities whenever an element appears in both. For example, if $B_1 = \{1,2,2,3\}$ and $B_2=\{2,3,3\}$, then the seven-element sum is $B_1+B_2 = \{1,2,2,2,3,3,3\}$. These concepts are essential to understand how leaves -- vertices of degree one -- dictate the appearance of balanced relations on tree graphs. 

We first introduce the distance of $c\in V$ to leaves, which is the multiset
\begin{equation}  
\label{eq: Lc}
L_c=\{\dist(c,l): l \text{ is a leaf of } G\},
\end{equation}
and also the multisets $L_c^m = \{\delta\in L_c: \delta\leq m\}$ and $L_c^m\oplus 1 = \{\delta+1:\delta\in L_c^m\}$.We then have:
\begin{prop}
 \label{prop:distleaf}
    Let $G$ be a tree graph, $c,d\in V$ and $\bowtie$ a non-trivial balanced relation on $G$. If $c\bowtie d$, then $L_c = L_d$.
\end{prop}

\begin{proof}
    We prove by induction on $m$ that, for every $c,d \in V$ such that $c \bowtie d$, one has $L_c^{m} = L_d^{m}$.
    For $m = 1$, the relation $c \bowtie d$ implies that $c$ is adjacent to the same number of leaves as $d$, and so $L_c^1 = L_d^1$.
    Assuming true for $m$, we prove for $m+1$. If $k_c$ and $k_d$ are the numbers of elements in $L_c^{m+1}$ and $L_d^{m+1}$ respectively, then 
    $$L_c^{m+1} =\min_{k_c}\Big(\sum_{v\sim c} L_v^m\oplus 1\Big) \hspace{1cm} L_d^{m+1} =\min_{k_d}\Big(\sum_{v\sim d} L_v^m\oplus 1\Big), $$
    where $\min_i(B)$ denotes the multiset consisting of the $i$ smallest elements of the multiset $B$, counted with multiplicities.
    Since $c\bowtie d$, we can associate bijectively each $v\sim c$ to a $v'\sim d$ such that $v\bowtie v'$. So, by the induction hypothesis 
    $$L_c^{m+1} = \min_{k_c}\Big(\sum_{v\sim c} L_v^m\oplus 1\Big) = \min_{k_c}\Big(\sum_{v\sim d} L_v^m\oplus 1\Big).$$

    The last step is to show that $k_d=k_c$. Without loss of generality, assume $k_d>k_c$. Then there exists at least one more leaf $l$ that has distance $m+1$ from $d$ than those from $c$. Let $d'$ be the vertex adjacent to $d$ on the unique path between $d$ and $l$ and $c'$ the bijective correspondence of $d'$ adjacent to $c$ such that $c'\bowtie d'$. Then, there exists at least one more leaf that has distance $m$ of $d'$ than those of $c'$, which is a contradiction since $L_{c'}^m = L_{d'}^m$ by the induction hypothesis. Therefore, $k_d=k_c$, which yields $L_c^{m+1} = L_d^{m+1}$. 
\end{proof}

A fundamental property of trees is that, for any two vertices, there exists a unique path connecting them (see, for example, \cite[Theorem 1.5.1]{Distel}). This feature dialogues with balanced relations on trees for vertices on the same $\bowtie$-classes. The phenomenon is formalized in Propositions~\ref{prop:reflected}, \ref{prop:samecolorleaf} and is illustrated in Fig.~\ref{fig:tree_balanced}~(left). 

\begin{defi}
    Let $p = (v_0,\cdots,v_k)$ be a path of vertices on a graph $G$ and $\bowtie$ a balanced coloring on the graph. We say that $p$ is \textit{$\bowtie$-reflected} if $v_i \bowtie v_{k-i}$ for every $i \in \{0,\cdots,k\}$. 
\end{defi}

\begin{prop}
\label{prop:reflected}
    Let $G$ be a tree graph and $\bowtie$ a non-trivial balanced coloring on $G$. For every two vertices in the same $\bowtie$-class, the path between them is $\bowtie$-reflected.
\end{prop}

\begin{proof}
    Let $c,d\in V$ such that $c\bowtie d$ and  consider the path $p = (c=v_0,v_1,\cdots,v_{k-1},v_k=d)$. Suppose that $\dist(c,d) = k > 2$, otherwise the result follows immediately because there will be either none or only one vertex between $c$ and $d$. 
    Assume by contradiction that $p$ is not $\bowtie$-reflected, i.e., for some $i \in \{1,\cdots,k-1\}$, $v_i \not\bowtie v_{i-i}$. Consider $i_0$ the first index such that this occurs. Since $\bowtie$ is a balanced relation, there exit vertices $v_{i_0-1}^{(1)}, \cdots, v_{k-i_0}^{(1)}$ such that $v_{i_0-1} \sim v_{k-i_0}^{(1)}$ and
    \begin{equation}
    \label{eq:reflected}
        \begin{cases}
        v_{j}^{(1)}\sim v_{j+1}^{(1)}, \; j\in\{i_0-1,\cdots,k-i_0-1\}\\[0.2cm]
        
        v_j^{(1)} \bowtie v_{j}, \; j\in\{i_0-1,\cdots, k-i_0\}.
    \end{cases}
    \end{equation}
    These vertices are all distinct from the vertices defined in $p$, otherwise the tree would have a cycle.

    Since $\bowtie$ is balanced, there must exist $v_{i_0-1}^{(2)}, \cdots, v_{k-i_0}^{(2)}$ such that $v_{i_0-1}^{(1)} \sim v_{k-i_0}^{(2)}$ and do satisfy the coupling/equivalence properties in (\ref{eq:reflected}). Again, these vertices are all distinct from the previously defined vertices, otherwise the tree would have a cycle. The picture of Fig.~\ref{fig:process} illustrates this process.
    
    To respect the balanced coloring, the same process must be repeated. Since the graph is finite, the procedure stops at some iteration step $s \in \N$, that is, for some $r\in\{i_0-1,\cdots,k-i_0\}$, $v_r^{(s)}$ is equal to one of the vertices previously defined. But this is a contradiction because $G$ has no cycles.

    \begin{figure}[htbp]
        \centering
        \begin{tikzpicture}[node distance={1.4cm}, thick, main/.style = {draw, circle, minimum size=0.8cm, inner sep=0cm}] 
            \node[main, fill=blue] (1) {$c$}; 
            \node[main, fill=orange] (2) [right of=1] {$v_1$};
            \node (3) [right of=2] {$\cdots$};
            \node[main, fill=purple] (4) [right of=3] {$v_{i_0-1}$}; 
            \node[main, fill=pink] (5) [right of=4] {$v_{i_0}$}; 
            \node (6) [right of=5] {$\cdots$};
            \node[main, fill=yellow] (7) [right of=6] {$v_{k-i_0}$};
            \node[main, fill=purple, font=\scriptsize] (8) [right of=7] {$v_{k-i_0-1}$};
            \node (9) [right of=8] {$\cdots$};
            \node[main, fill=orange] (10) [right of=9] {$v_{k-1}$};
            \node[main, fill=blue] (11) [right of=10] {$d$};

            \node[main, fill=yellow] (12) [above of=5] {$v_{k-i_0}^{(1)}$};
            \node (13) [right of=12] {$\cdots$};
            \node[main, fill=pink] (14) [right of=13] {$v_{i_0}^{(1)}$};
            \node[main, fill=purple] (15) [right of=14] {$v_{i_0-1}^{(1)}$};

            \node[main, fill=yellow] (16) [above of=14] {$v_{k-i_0}^{(2)}$};
            \node (17) [left of=16] {$\cdots$};
            \node[main, fill=pink] (18) [left of=17] {$v_{i_0}^{(2)}$};
            \node[main, fill=purple] (19) [left of=18] {$v_{i_0-1}^{(2)}$};
            
            \draw (1) -- (2);
            \draw (2) -- (3);
            \draw (3) -- (4);
            \draw (4) -- (5);
            \draw (5) -- (6);
            \draw (6) -- (7);
            \draw (7) -- (8);
            \draw (8) -- (9);
            \draw (9) -- (10);
            \draw (10) -- (11);

            \draw (4) -- (12);
            \draw (12) -- (13);
            \draw (13) -- (14);
            \draw (14) -- (15);

            \draw (15) -- (16);
            \draw (16) -- (17);
            \draw (17) -- (18);
            \draw (18) -- (19);

        \end{tikzpicture}
        \caption{A schematic representation of the coupling process in the proof of Proposition~\ref{prop:reflected}.}
        \label{fig:process}
    \end{figure}
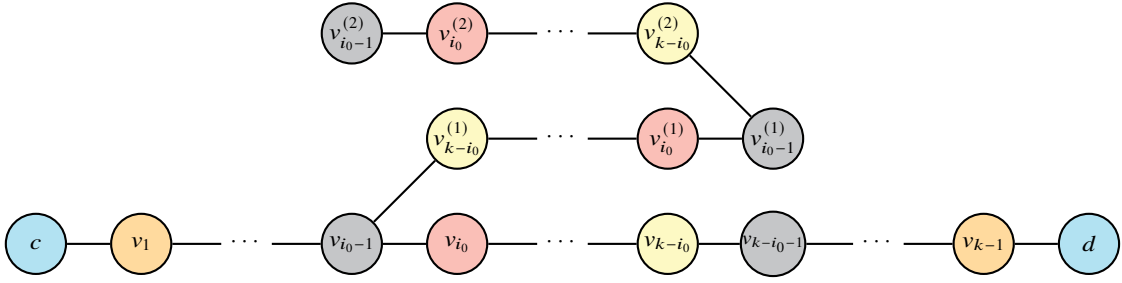

\end{proof}

\begin{cor}
\label{no_vertex_middle}
    Let $\bowtie$ be a non-trivial balanced relation on a tree $G$ and $c,d\in V$. If $c\bowtie d$ then no vertex in the path between $c$ and $d$ belongs to $[c]_{\bowtie}$.
\end{cor}

\begin{proof}
    Consider the path $p = (c,v_1,\cdots,v_{k-1},d)$ and suppose, by contraction, that there is $i \in \{ 1, \ldots, k-1\}$ with  $v_i\in[c]_{\bowtie}$. First, we notice that $i\neq 1,k-1$, that is, $v_i$ is not adjacent to $c$ or $d$; otherwise, all vertices of $p$ would be in $[c]_{\bowtie}$ by Proposition~\ref{prop:reflected}. By definition of a balanced relation, this would generate an infinite path whose vertices would belong to $[c]_{\bowtie}$. Therefore $k\geq 4$.

    Using Proposition~\ref{prop:reflected} three times -- for the path $p$, for the  path between $c$ and $v_i$ and for the path between $v_i$ and $d$ -- it follows that $v_1 \bowtie v_{k-1}$, $v_1 \bowtie v_{i-1}$ and $v_{i+1} \bowtie v_{k-1}$. By transitivity, $v_{i-1} \bowtie v_{i+1}$ and so there exists a new vertex $v_{i-1}^{(1)}\in V$ adjacent to $c$ such that $v_{i-1}^{(1)}\in[v_{i-1}]_{\bowtie}$. Now, as $G$ has no cycles and to respect the balanced relation, $v_{i-1}^{(1)}$ must receive an input from a new vertex in $[v_2]_{\bowtie}$, and so on, as in the argument presented in the proof of Proposition~\ref{prop:reflected}. This would give an infinite path on $G$, which is a contradiction. 
\end{proof}

A natural question is whether a tree can admit a nontrivial balanced relation in which all leaves belong to distinct $\bowtie$-classes. The following result shows that this is not possible.

\begin{prop}
\label{prop:samecolorleaf}
    If a tree graph $G$ admits a non-trivial balanced relation, then at least two leaves of $G$ are in the same $\bowtie$-class.
\end{prop}

\begin{proof}
Let $c,d\in V$ be distinct, such that $c \bowtie d$, and let $p$ be the path between them. If $c$ and $d$ are leaves, then the argument is complete. As a leaf cannot be $\bowtie$-equivalent to a non-leaf, we now suppose that $c$ and $d$ are not leaves. Since $\deg(c),\deg(d)\geq2$, there exists $c_1 \sim c$ with $c_1 \notin p$. We claim that there must exist $d_1\sim d$, with $d_1\not\in p$ such that $c_1 \bowtie d_1$. Indeed, given that $\bowtie$ is balanced, if $p = (c,d)$ then the statement follows immediately, otherwise there exists $\beta\in \mathcal{B}(c,d)$ that maps the vertex adjacent to $c$ on $p$ to the vertex adjacent to $d$ on $p$, as $p$ is $\bowtie$-reflected by Proposition~\ref{prop:reflected}. As $d_1\notin p$, $d_1$ is distinct of $c_1$ by the absence of cycles on $G$. Therefore, if $c_1$ is a leaf, so is $d_1$, and the result follows.

Now, suppose that $c_1$ is not a leaf and let $l$ be a leaf at the end of a path $p_0 = (c,c_1,\cdots,c_{k-1},l)$ that runs through $c$ and $c_1$, where $k=\dist(c,l)$. As $\bowtie$ is balanced and $c_1\bowtie d_1$, there exists $d_2 \in V$ distinct of $c_2$ (otherwise $G$ would have a cycle) such that $d_2\sim d_1$ and $d_2\bowtie c_2$. Again, since $d_2\bowtie c_2$, there exists $d_3 \in V$ such that $d_3\sim d_2$, $d_3\bowtie c_3$ and $d_3\neq c_3$. By this argument, we can construct a path $p_0' = (d,d_1,\cdots,d_{k-1})$ such that, for all $i\in\{1,\cdots,k-1\}$, $d_i\bowtie c_i$. By Proposition~\ref{prop:distleaf}, and also by the facts that $\bowtie$ is balanced and $c_{k-1} \bowtie d_{k-1}$, there exists at least one leaf $l' \in V$ such that $l' \sim d_{k-1}$ and $l'\bowtie l$.
\end{proof}

Path graphs are the simplest class of trees and represent linear chains of connections -- two leaves and all other vertices of degree two. We use the results above to prove the absence of exotic synchronies in this case:

\begin{cor}
\label{path}
    A path graph does not admit exotic balanced relation.
\end{cor}

\begin{proof}
    Let $\bowtie$ be a non-trivial balanced relation on the path graph $P_n=(v_1,\cdots,v_n)$ of order $n$. By Proposition~\ref{prop:samecolorleaf} $v_1\bowtie v_n$, and by Proposition~\ref{prop:reflected} $P_n$ is $\bowtie$-reflected. As $\Aut(P_n) = \{id, \phi\}$, where $\phi$ maps $v_i$ to $v_{n-i+1}$ for every $i \in \{1,\cdots,n\}$, we have that $\Delta_{\bowtie} = \Fix(\Aut(P_n))$ and the result follows.
\end{proof}

\subsection{Fixed-point coloration on trees}

Recall that a connected graph $G$ of order $n$ is a tree if and only if $G$ has $n-1$ edges. We use this characterization to present a pruning process on a tree graph, which generates a new tree and preserves balanced relations. If $L$ is the set of the leaves of $G$, we denote by $G\setminus L$ the graph obtained by removing the vertices of $L$ (and all edges incident to the vertices in $L$).

\begin{prop}
    If $\bowtie$ a balanced relation on a tree $G$, then $G\setminus L$ is a tree and $\left. \bowtie\right|_{G\setminus L}$ is a balanced relation on $G\setminus L$.
\end{prop}

\begin{proof}
    $G$ has $|V|-1$ edges and each leaf of $G$ is connected to the rest of the tree by one edge. Then, by removing $L$ of $G$ we also remove $|L|$ edges of $G$. But $G\setminus L$ is connected with $|V|-|L|$ vertices and $|V|-|L|-1$ edges, so $G\setminus L$ is a tree.
    
    For the second claim, if $c\bowtie d$ then there exists an input isomorphism $\beta: I(c)\to I(d)$ such that $\mathcal{T}(e) \bowtie \mathcal{T}(\beta (e))$ for all $e\in I(c)$. If $\mathcal{T}(e) \in L$, then $\mathcal{T}(\beta (e)) \in L$, because a leaf can only be balanced equivalent to a leaf. Therefore, $\left. \beta\right|_{G\setminus L}$ is an input isomorphism such that $\mathcal{T}(e) \bowtie \mathcal{T}(\beta (e))$ for all edges $e$ in the input set of $c$ in the network $G\setminus L$. 
\end{proof}

The pruning process above can be repeated by removing the leaves of $G\setminus L$, and so on. The procedure ends at step $s$, say,  when the remaining graph is either the 2-vertex graph (and one edge) or the 1-vertex graph (and no edge), since the next step would result in the empty graph. Let
\begin{equation}
\label{eq:pruning}
G^0 = G ,  \ L^0=L, \  G^i = G^{i-1}\setminus L^{i-1},  \ \  i\in\{0,\cdots,s\}, 
\end{equation} where $L^{i-1}$ denotes the set of the leaves of the tree $G^{i-1}$. This process is a method to obtain a balanced class via leaves of an iterated graph in an algorithmic way: 

\begin{lem}
\label{lem:classcolor}
    If $\bowtie$ is a non-trivial balanced relation on a tree $G$ then for each $c \in V$ there exists an index $i\in\{0,\cdots,s\}$ such that all vertices of $[c]_{\bowtie}$ are leaves of $G^i$.
\end{lem}

\begin{proof}
    It is direct from Proposition~\ref{prop:distleaf}.
\end{proof}

Proposition~\ref{prop:samecolorleaf} and Corollary~\ref{path} naturally lead to the question of whether trees admit exotic relations. With the previous lemma at hand, we are now in a position to state the main theorem of this section.

\begin{theorem}
\label{thm:no_exotic_trees}
    Every balanced coloring on a tree $G$ is a fixed-point coloration.
\end{theorem}

\begin{proof}
    The proof is based on the tree pruning process defined in (\ref{eq:pruning}) and schematically illustrated in  Fig.~\ref{fig:no_exotic_thm}. Let $\bowtie$ be a balanced relation on a tree $G$. Suppose that $\bowtie$ is non-trivial, otherwise $\Delta_{\bowtie} = \Fix(\{id\})$ and we are done. It suffices to construct an automorphism $\phi: V \to V$ such that $\Delta_{\bowtie} = \Fix(\langle\phi\rangle)$. Let $i_0\in \{0,\cdots,s\}$ be the smallest integer such that $\left. \bowtie\right|_{G^{i_0}}$ is a trivial balanced relation on $G^{i_0}$ and define $\phi_0$ as the trivial permutation of $V(G^{i_0})$. Therefore, $\phi_0$ is an automorphism of $G^{i_0}$ and $\Delta_{\left. \bowtie\right|_{G^{i_0}}} = \Fix(\langle\phi_0\rangle)$. 

    Denote $i_1=i_0-1$. Since $\left. \bowtie\right|_{G^{i_1}}$ is a non-trivial balanced relation on $G^{i_1}$, let
    $$[v_1^1]_{\left. \bowtie\right|_{G^{i_1}}} = \{v_1^1,\cdots,v_1^{q_1}\}, \cdots, [v_{j_1}^1]_{{\left. \bowtie\right|_{G^{i_1}}}}= \{v_{j_1}^1,\cdots,v_{j_1}^{q_{j_1}}\}$$
    be all new equivalence classes of $\left. \bowtie\right|_{G^{i_1}}$, i.e., the equivalence classes not listed in $\left. \bowtie\right|_{G^{i_0}}$. By Lemma~\ref{lem:classcolor}, for each $r\in\{1,\cdots,j_1\}$, $[v_r^1]_{{\left. \bowtie\right|_{G^{i_1}}}} = [v_r^1]_{\bowtie}$. Now, let $\phi_1$ be the extension of $\phi_0$ to $V(G^{i_1})$ by adding, for each $r\in\{1,\cdots,j_1\}$, the permutation cycle $(v_r^1\, v_r^2\, \cdots \, v_r^{q_r})$ to $\phi_0$.

    As $\left. \bowtie\right|_{G^{i_0}}$ is trivial balanced relation, for each $r\in\{1,\cdots,j_1\}$, all elements of $[v_r^1]_{\left. \bowtie\right|_{G^{i_1}}}$ must be adjacent to a unique vertex of $G^{i_0}$, since $\left. \bowtie\right|_{G^{i_1}}$ is a balanced relation on $G^{i_1}$. Therefore, $\phi_1$ is a permutation on $V(G^{i_1})$ that preserves adjacency, i.e., it is an automorphism of $G^{i_1}$. Moreover, the equality $\Delta_{\left. \bowtie\right|_{G^{i_1}}} = \Fix(\langle\phi_1\rangle)$ follows directly by construction. 

        \begin{figure}
        \centering
        \begin{tikzpicture}[x=0.75pt,y=0.75pt,yscale=-1,xscale=1, node distance={1.4cm}, thick, main/.style = {draw, circle, minimum size=0.8cm, inner sep=0cm}]
        
        %Shape
        \draw   (272,75) .. controls (322,55) and (362,60) .. (381,74) .. controls (400,88) and (409,133) .. (379,141) .. controls (349,149) and (345,124) .. (292,134) .. controls (239,144) and (222,95) .. (272,75) -- cycle ;

        \draw   (181,37) .. controls (241,1) and (375,-24) .. (439,15) .. controls (503,54) and (516,163) .. (458,187) .. controls (400,211) and (340,124) .. (237,168) .. controls (134,212) and (121,73) .. (181,37) -- cycle ;
        
        % Texto
        \draw (312,92) node [anchor=north west][inner sep=0.75pt]   [align=left] {$G^{i_0}$};
        \draw (388,162) node [anchor=north west][inner sep=0.75pt]   [align=left] {$G^{i_1}$};
        
        % Pontos na curva
        \coordinate (a) at (242,105);
        \coordinate (b) at (398,105);

        %Classe 1
        \node[main, fill=blue] (1) at (190,140) {$v_1^1$};
        \node (2) at (190,100) {$\vdots$};
        \node[main, fill=blue] (3) at (190,70) {$v_1^{q_1}$};

        \draw (1) -- (a);
        \draw (3) -- (a);

        %Classe ...
        \node[font=\Large] at (315,40) {$\cdots$};

        %Classe j_1
        \node[main, fill=pink] (4) at (450,70) {$v_{j_1}^1$};
        \node (5) at (450,100) {$\vdots$};
        \node[main, fill=pink] (6) at (450,140) {$v_{j_1}^{q_{j_1}}$};

        \draw (4) -- (b);
        \draw (6) -- (b);

        %Classes ligadas ao azul
        \node[main, fill=yellow] (7) at (170,210) {$v_{j_1+1}^1$}; 
        \node (8) at (140,185) {$\ddots$};
        \node[main, fill=yellow, scale=0.75] (9) at (120,160) {$v_{\scriptsize j_1+1}^{\scriptsize (a-1)q_1+1}$};

%(130,180)
        
        \node (10) at (120,100) {$\vdots$};

        \node[main, fill=yellow] (11) at (120,50) {$v_{j_1+1}^{q_1}$};
        \node (12) at (140,15) {\reflectbox{$\ddots$}};
        \node[main, fill=yellow] (13) at (170,0) {$v_{j_1+1}^{aq_1}$};

        \draw (7) -- (1);
        \draw (9) -- (1);
        \draw (11) -- (3);
        \draw (13) -- (3);

        \end{tikzpicture}
        \caption{A schematic representation of the coupling process in the proof of Theorem~\ref{thm:no_exotic_trees}.}
        \label{fig:no_exotic_thm}
    \end{figure}
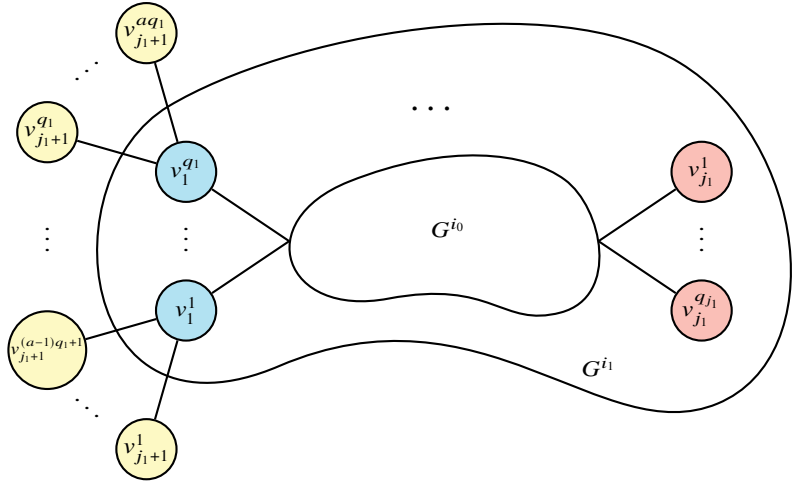
    
    If $G^{i_1} \neq G$ we repeat the process: define $i_2=i_1-1$ and let 
    $$[v_{j_1+1}^1]_{\left. \bowtie\right|_{G^{i_2}}} = \{v_{j_1+1}^1,\cdots,v_{j_1+1}^{q_{j_1+1}}\}, \cdots, [v_{j_1 +j_2}^1]_{{\left. \bowtie\right|_{G^{i_2}}}}= \{v_{j_1+j_2}^1,\cdots,v_{j_1+j_2}^{q_{j_1+j_2}}\}$$
    be all new equivalence classes of $\left. \bowtie\right|_{G^{i_2}}$ (those not listed in $\left. \bowtie\right|_{G^{i_1}}$). Again, by Lemma~\ref{lem:classcolor}, for each $r\in\{j_1+1,\cdots,j_1+j_2\}$, we have $[v_r^1]_{{\left. \bowtie\right|_{G^{i_2}}}} = [v_r^1]_{\bowtie}$. All the elements of a new equivalence class when looked at as vertices of $G^{i_2}$ are adjacent only to vertices of a unique equivalence class among $[v_1^1]_{\left. \bowtie\right|_{G^{i_1}}}, \cdots,[v_{j_1}^1]_{{\left. \bowtie\right|_{G^{i_1}}}}$; otherwise, this would contradict the fact that $\left. \bowtie\right|_{G^{i_2}}$ is balanced as no $\left. \bowtie\right|_{G^{i_1}}$-class could appear in a step back of the iteration. 

    In order to respect the facts that $\left. \bowtie\right|_{G^{i_2}}$ is balanced and $G^{i_2}$ has no cycles, if there is a number  $a$ of vertices in the class $[v_r^1]_{\left. \bowtie\right|_{G^{i_2}}}$, $r\in\{j_1+1,\cdots,j_1+j_2\}$, that are adjacent to $v_b^1$, for some $b\in\{1,\cdots,j_1\}$, then each vertex of the class $[v_b^1]_{\left. \bowtie\right|_{G^{i_1}}}$ must be adjacent to $a$  vertices of $[v_r^1]_{\left. \bowtie\right|_{G^{i_2}}}$ without repetition, that is, $q_r = a  q_b$. Without loss of generality, relabeling the indices if necessary, suppose that $v_r^1,v_r^{q_b+1},\cdots, v_r^{(a-1)q_b+1}$ are the vertices adjacent to $v_b^1$, that $v_r^2,v_r^{q_b+2},\cdots, v_r^{(a-1)q_b+2}$ are the vertices adjacent to $v_b^2$ and so on, i.e., that $v_r^{q_b},v_r^{2q_b},\cdots, v_r^{aq_b}$ are the vertices adjacent to $v_b^{q_b}$. With this notation, construct the permutation cycle given by
    \begin{equation}
    \label{eq:cycle}
    (v_r^1\, v_r^2\, \cdots v_r^{q_b} \, v_r^{q_b+1} \, \cdots \, v_r^{2q_b} \, v_r^{2q_b +1} \cdots \, v_r^{aq_b}).  
    \end{equation}
    This cycle exhausts the vertices of the equivalence class $[v_r^1]_{\left. \bowtie\right|_{G^{i_2}}}$. 
    
    Let $\phi_2$ be an extension of $\phi_1$ to $V(G^{i_2})$ by adding for each $r\in\{j_1+1,\cdots,j_1+j_2\}$ the permutation cycle described in (\ref{eq:cycle}) with its respective $b\in\{1,\cdots,j_1\}$. By construction, $\phi_2$ preserves adjacency on $G^{i_2}$ (it is an automorphism of $G^{i_2}$) and $\Delta_{\left. \bowtie\right|_{G^{i_2}}} = \Fix(\langle\phi_2\rangle)$. If $G^{i_2} = G$, then the argument is complete. Otherwise, it is sufficient to repeat the $i_2$-step described above $i_0-2$ more times.
\end{proof}

\begin{cor}
    If $G$ is an asymmetric tree, then $G$ does not admit a non-trivial balanced coloring.
\end{cor}

We end this subsection with a curious fact about synchronization in tree-type graphs, as a consequence of  Lemma~\ref{lem:classcolor}. This says that two vertices of the same $\bowtie$-class are adjacent if and only if the equivalence class is fully determined by them and the tree pruning process ends with these two vertices. We have:

\begin{prop}
\label{prop:adjacent_color}
    Let $\bowtie$ be a non-trivial balanced coloring on a tree $G$ and $c,d\in V$ such that $c\bowtie d$. Then $c$ and $d$ are adjacent if and only if $V(G^s) =[c]_{\bowtie} = \{c,d\}$, where $G^s$ is the tree of step $s$ in (\ref{eq:pruning}).
\end{prop}

\begin{proof}
    Suppose that $c$ and $d$ are adjacent. If there is a vertex $v\in V$ in $[c]_{\bowtie}$, then either $c$ is in the path between $d$ and $v$ or $d$ is in the path between $c$ and $v$, and both contradict Corollary~\ref{no_vertex_middle}. Therefore, $[c]_{\bowtie} = \{c,d\}$. For the first equality, by Lemma~\ref{lem:classcolor} there exists an index $i \in \{0,\cdots,s\}$ such that the vertices of $[c]_{\bowtie}$ are leaves of $G^i$. As $G^i$ is a tree with two adjacent leaves, then $i=s$ and so $V(G^s) = \{c,d\}$. The converse is trivial, since the two vertices on the last pruning iteration are adjacent.
\end{proof}

\subsection{Quotient Networks of Trees}

Quotient network of an undirected network is not necessarily undirected; in fact, it is often directed. So, in this subsection, we adopt the following definition of a cycle in a directed graph: it is a non-empty path that starts and ends at the same vertex, following the direction of the edges, and such that all other vertices (and edges) in the path appear exactly once.

\begin{defi}
    Let $G$ be a directed graph with one type of edge. The \textit{undirected simplification} $G^*$ of $G$ is a simple undirected graph having the same set of vertices as $G$ and whose edges are defined as follows: two vertices $c$ and $d$ are joined by an undirected edge whenever there is a 2-cycle connecting them in $G$.
\end{defi}

\begin{prop} \label{prop:G*}
    Let $G$ be a tree graph, $\bowtie$ a balanced relation on $G$ and $G_{\bowtie}$ the quotient network. Then $G_{\bowtie}^*$ is a tree.
\end{prop}

\begin{proof}
    We show that $G_{\bowtie}^*$ is connected and acyclic. Connectivity is trivial as the quotient of an undirected connected network is connected. This follows directly from the fact that if $c\sim d$ in $G$, with $c$ and $d$ in distinct $\bowtie$-classes, then there exists a 2-cycle between $[c]_{\bowtie}$ and $[d]_{\bowtie}$. So,  $[c]_{\bowtie} \sim [d]_{\bowtie}$ in $G_{\bowtie}^*$ and every path in $G$ can be seen as a path in $G_{\bowtie}^*$. 

    Suppose by contradiction that $G_{\bowtie}^*$ has a cycle, $([c_0]_{\bowtie},[c_1]_{\bowtie}, \cdots, [c_k]_{\bowtie}, [c_0]_{\bowtie})$ say. By construction of $G_{\bowtie}^*$, this is also a cycle in the quotient $G_{\bowtie}$. We claim that $k\geq 2$. Indeed, $k\neq 0$ since no vertex $[c]_{\bowtie}\in G_{\bowtie}^*$ admits a loop. Moreover, $k=1$ would imply the existence of two undirected edges between $[c_0]_{\bowtie}$ and $[c_1]_{\bowtie}$ in $G_{\bowtie}^*$, contradicting that $E_{G_{\bowtie}^*}$ is a set and not a multiset.
    
    Therefore, if $v_0,v_{k+1} \in [c_0]_{\bowtie}$ are distinct vertices of the tree $G$, then there are vertices $v_1,\cdots,v_k$ of $G$, $k\geq 2$, such that $v_i\in[c_i]_{\bowtie}$, $i\in\{1,\cdots,k\}$. Let $p = (v_0, v_1, \ldots, v_k, v_{k+1})$ be the path between $v_0$ and $v_{k+1}$ -- unique since $G$ is a tree. However, each vertex on the path between $v_0$ and $v_{k+1}$ belongs to a distinct $\bowtie$-class, which contradicts Proposition~\ref{prop:reflected}, since $k \geq 2$ and $p$ is not $\bowtie$-reflected.
\end{proof}

\begin{rem}
    It follows from Proposition~\ref{prop:G*} that $G_{\bowtie}$ does not have cycles other than the ones formed by a 2-cycle between two adjacent classes. $G_{\bowtie}$ can have a 1-cycle (looping), and a 2-cycle appears in $G_{\bowtie}$ whenever two elements of distinct $\bowtie$-classes are adjacent on $G$ since $G$ is undirected, as illustrated in Fig.~\ref{fig:tree_balanced}. Moreover, the previous result guarantees that the path between any two nodes of $G_{\bowtie}$ is unique. However, when looking at edges in the path, it may happen that the path is not unique, since there may be more than one directed edge from one vertex to another.
\end{rem}

\begin{cor}
    Let $G$ be a tree graph, $\bowtie$ a balanced relation on $G$ and $G_{\bowtie}$ the quotient network. If $l$ is a leaf of $G$, then $[l]_{\bowtie}$ is a leaf of $G_{\bowtie}^*$. 
\end{cor}

\begin{proof}
    This follows directly from the fact that the degree of a vertex is preserved in the quotienting operation, so if $l$ is a leaf of $G$ then $[l]_{\bowtie}$ has degree one in $G_{\bowtie}$ and, consequently, it has degree one on $G_{\bowtie}^*$. 
\end{proof}

%%%%%%%%%%%%%%%% SECTION 4 %%%%%%%%%%%%%%%%%%%%%%%%%%%%%%%%%%

\section{Admissible Vector Fields on Trees}
\label{sec:4}

Tree structure imposes restrictions on admissible vector fields that contribute to the study of stability of equilibrium points or synchrony subspaces. In Subsection~\ref{subsec:4.1}, we discuss linear stability of admissible vector fields and in Subsection~\ref{subsec:4.2} we explore the importance of cherries -- groups of pendant vertices adjacent to a common neighbor -- in this characterization. Motivated by the relevance of this structure, we further focus, in Subsection~\ref{subsec:4.3}, on the study of how they dictate admissible maps and the Lyapunov stability of the synchrony subspaces induced by them.

\subsection{Linear Stability}
\label{subsec:4.1}

Throughout this subsection, we assume that the phase space associated with each cell is one-dimensional. In an undirected network with a single arrow type, admissibility implies that the vector field component associated with a cell depends only on its degree. Hence, cells with equal degree have identical vector field components. Under these conditions, if $f:P\to P$ is an admissible vector field with an equilibrium at the origin, then for every cell $c$, the $c$ component of the linearization $d_0f$ at the origin is
\begin{equation}
\label{eq:linearization}
(d_0f)_c(x) = \alpha_{\deg(c)} x_c + \beta_{\deg(c)}\sum_{v\sim c}x_v. 
\end{equation}
For each degree \(i\), the corresponding cells are associated with a pair \((\alpha_i,\beta_i)\) of coefficients governing the respective components of \(d_0 f\). As we work with networks of 1-cell type, we assume for consistency that all internal dynamics coefficients coincide, and we denote this common coefficient by $\alpha$. In what follows, we discuss the occurrence of $\alpha$ as an eigenvalue of $d_0 f$. We start with an example.

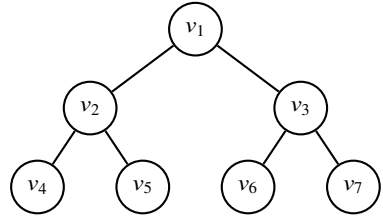
\begin{figure}[htbp]
    \centering
    \begin{tikzpicture}
    \node at (0,0) {\begin{tikzpicture}[node distance={2cm}, thick, main/.style = {draw, circle, minimum size=0.7cm}, scale=0.7]
            
            \node[main] (1) at (0, 0) {$v_1$};
            \node[main] (2) at (-2, -1.5) {$v_2$};
            \node[main] (3) at (2, -1.5) {$v_3$};
            \node[main] (4) at (-3, -3) {$v_4$};
            \node[main] (5) at (-1, -3) {$v_5$};
            \node[main] (6) at (1, -3) {$v_6$};
            \node[main] (7) at (3, -3) {$v_7$};

            \draw (1) -- (2);
            \draw (1) -- (3);
            \draw (2) -- (4);
            \draw (2) -- (5);
            \draw (3) -- (6);
            \draw (3) -- (7);
     
        \end{tikzpicture}};
    \end{tikzpicture}
    \caption{The binary tree of order seven.}
    \label{fig:binary}
\end{figure}

\begin{ex}
\label{ex:binary_tree}
Consider the binary tree graph of order seven given in  Fig.~\ref{fig:binary}. The matricial form of the linearization  $d_0f$   for this network is
$$[d_0f]=\begin{pmatrix}
    \alpha & \beta_2 & \beta_2 & 0 & 0 & 0 & 0\\
    \beta_3 & \alpha & 0 & \beta_3 & \beta_3 & 0 & 0\\
    \beta_3 & 0 & \alpha & 0 & 0 & \beta_3 & \beta_3\\
    0 & \beta_1 & 0 & \alpha & 0 & 0 & 0\\
    0 & \beta_1 & 0 & 0 & \alpha & 0 & 0\\
    0 & 0 & \beta_1 & 0 & 0 & \alpha & 0\\
    0 & 0 & \beta_1 & 0 & 0 & 0 & \alpha
\end{pmatrix}.$$

The coefficient $\alpha$ is an eigenvalue of $d_0f$ of multiplicity at least two, since $u_1 = (0,0,0,1,-1,0,0)$ and $u_2=(0,0,0,0,0,1,-1)$ are eigenvectors associated with  $\alpha$. 
Now, for $a_1 = (0,1,-1,0,0,0,0)$ and $a_2 = (0,0,0,1,1,-1,-1)$, we have
    \begin{align*}
        d_0f(a_1) &= (0,\alpha,-\alpha,\beta_1,\beta_1,-\beta_1,-\beta_1) = \alpha a_1 + \beta_1 a_2\\
        d_0f(a_2) &= (0,2\beta_3,-2\beta_3,\alpha,\alpha,-\alpha,-\alpha) = 2\beta_3 a_1 + \alpha a_2,
    \end{align*}
and so $\sp\{a_1,a_2\}$ is an invariant subspace under $d_0f$. Hence, the eigenvalues of
$$\begin{pmatrix}
    \alpha & \beta_1\\
    2\beta_3 & \alpha
\end{pmatrix}$$
are the eigenvalues of $d_0f$, which are
    $$\lambda_1 = \alpha + \sqrt{2\beta_1\beta_3}, \qquad \lambda_2 = \alpha - \sqrt{2\beta_1\beta_3}. $$
For $b_1 = (1,0,0,0,0,0,0)$, $b_2 = (0,1,1,0,0,0,0)$ and $b_3 = (0,0,0,1,1,1,1)$, we have
    \begin{align*}
        d_0f(b_1) &= (\alpha,\beta_3,\beta_3,0,0,0,0) = \alpha b_1 + \beta_3 b_2\\
        d_0f(b_2) &= (2\beta_2,\alpha,\alpha,\beta_1,\beta_1,\beta_1,\beta_1) = 2\beta_2 b_1 + \alpha b_2 + \beta_1 b_3\\
        d_0f(b_3) &= (0,2\beta_3,2\beta_3,\alpha,\alpha,\alpha,\alpha) = 2\beta_3b_2 + \alpha b_3.
    \end{align*}
Since $u_1,u_2,a_1,a_2\notin \sp\{b_1,b_2,b_3\}$, the eigenvalues of 
    $$\begin{pmatrix}
        \alpha & \beta_3 & 0\\
        2\beta_2 & \alpha & \beta_1\\
        0 & 2\beta_3 & \alpha
    \end{pmatrix}$$
are the remaining eigenvalues of $d_0f$, which are
    $$\lambda_3 = \alpha, 
    \qquad \lambda_4 = \alpha + \sqrt{2\beta_3(\beta_1+\beta_2)}, \qquad \lambda_5 = \alpha -\sqrt{2\beta_3(\beta_1+\beta_2)}. $$
 %   This process exhausts all the eigenvalues of $d_0f$ and we 
 % have constructed $\spec(d_0f)$. 
\end{ex}

Let us now expand from this example using the notion of a matching of a graph. For a general tree network of order $n$ with cells $v_1,\cdots,v_n$, we have that 
\begin{equation} \label{eq: d_0f}
d_0f = DA + \alpha Id,
\end{equation}
where $D = \operatorname{diag}(\beta_{\deg(v_1)}, \cdots, \beta_{\deg(v_n)})$. 
In graph theory, a \textit{matching} $M$ of a graph is a set of edges such that no two edges of $M$ share common vertices (see, for example, \cite[Subsection 3.5]{Godsil}). Also, $M$ is a \textit{perfect matching} if every vertex of the graph is an endpoint of some edge in $M$, and $M$ is a \textit{maximum matching} if it is a matching with the largest number of edges, $\upnu(G)$ say. A classical spectral theorem states that if $G$ is tree graph of order $n$ then $0$ is an eigenvalue of the adjacent matrix $A$ with multiplicity $n-2\upnu (G)$ (see, for example, \cite{G}). This result is used to present a lower bound for the algebraic multiplicity of the eigenvalue $\alpha$:

\begin{prop}
\label{prop:alpha}
    Let $f$ be an admissible vector field on a tree network $G$ of order $n$ with an equilibrium point at the origin. Then the coefficient $\alpha$ of the internal cell dynamics is an eigenvalue of $d_0f$ with multiplicity at least $n-2\upnu(G)$.
\end{prop}

\begin{proof}
    Since $\ker(A) \subseteq \ker(DA)$, if $0$ is an eigenvalue of $A$ then $\alpha$ is an eigenvalue of $d_0f$, and the result follows by the spectral theorem for trees.
\end{proof}

If $\beta_{\deg(c)} \neq 0$ for every cell $c$, then $\ker(A)=\ker(DA)$. Moreover, if a graph $G$ of order $n$ admits a perfect matching, then $n=2\upnu(G)$. This yields a sufficient condition for the occurrence of $\alpha$ based on the topology of the tree:

\begin{cor}
    Under the conditions of Proposition~\ref{prop:alpha}, assume $\beta_{\deg(c)}$ nonzero for all cell $c$. If $G$ admits a perfect matching, then $\alpha$ is not an eigenvalue of $d_0f$.
\end{cor}

Let us further discuss the occurrence and multiplicity  of  $\alpha$ as a distinguished eigenvalue  of $d_0f$. 
From the equality $\spec(d_0f) =\{ \alpha + \lambda: \lambda\in \spec(DA)\}$, clearly  the algebraic multiplicity of  $\alpha$ is related to the algebraic multiplicity of $0$ as an eigenvalue of $DA$. Revisiting Example~\ref{ex:binary_tree}, on crossing the condition  $\beta_1 = -\beta_2$ to the generic condition $\beta_1 \neq -\beta_2$ the algebraic multiplicity of $\alpha$ decreases by two units. A broader question concerns whether, and in what manner, the algebraic multiplicity of $\alpha$ varies in a tree network as the coupling parameters change. The next two results address this issue.

\begin{prop}
    If $\lambda$ is an eigenvalue of $DA$ then so is $-\lambda$. 
\end{prop}

\begin{proof}
As $G$ is bipartite, we can relabel the vertices if necessary so we can write 
    $$DA = \begin{pmatrix}
        D_1 & 0\\
        0 & D_2
    \end{pmatrix}\begin{pmatrix}
        0 & B\\
        B^T & 0
    \end{pmatrix} = \begin{pmatrix}
        0 & D_1B\\
        D_2B^T & 0
    \end{pmatrix}.$$

    If $u = (X, Y)$ is an eigenvector of $DA$ associated with the eigenvalue $\lambda$ then $D_1BY = \lambda X$ e $D_2B^TX = \lambda Y$. Hence,
    $$DA \begin{pmatrix}
        X\\
        -Y
    \end{pmatrix} = \begin{pmatrix}
        -D_1BY\\
        D_2B^TX\end{pmatrix} = \begin{pmatrix}
        -\lambda X\\
        \lambda Y
        \end{pmatrix} =-\lambda \begin{pmatrix}
            X\\
            -Y
        \end{pmatrix}.
        $$
        \mbox{}
\end{proof}

We then have:

\begin{cor}
For any variation of the coupling parameters \(\beta_{\deg(v_i)}\), \(i = 1, \ldots, n\), if the algebraic multiplicity of the eigenvalue \(\alpha\) of \(d_0 f\) changes, then it changes by an even integer.
\end{cor}

We finish by presenting an estimate for the elements of $\spec(d_0f)$, which depends on the signs of the coupling parameters  $\beta_{\deg(c)}$.  For that,  given a matrix $B$ let us denote  $\lambda_{\min}(B)$ and $\lambda_{\max}(B)$ the smallest and the largest eigenvalues of $B$, respectively.

\begin{theorem}
    Consider the notation described above.  If $\beta_{\deg(c)} > 0$ for all cell $c$, then the elements of $\spec(d_0f)$ are real and, for all $\lambda\in\spec(d_0f)$, we have
    \begin{equation}
    \label{eq:beta_real}
        \alpha+\lambda_{\min}(D^{\frac{1}{2}} A D^{\frac{1}{2}})\leq\lambda\leq \alpha + \lambda_{\max}(D^{\frac{1}{2}} A D^{\frac{1}{2}}).
    \end{equation}
    If $\beta_{\deg(c)} \leq 0$ for some cell $c$, then for all $\lambda\in\spec(d_0f)$ we have
    \begin{equation}
        \label{eq:beta_neg}
        \alpha+\lambda_{\min}\Big(\frac{1}{2}(DA+AD)\Big)\leq\operatorname{Re}(\lambda)\leq \alpha + \lambda_{\max}\Big(\frac{1}{2}(DA+AD)\Big).
    \end{equation}
\end{theorem}

\begin{proof}
    Assume $\beta_{\deg(c)} > 0$ for all cell $c$. Considering the diagonal matrix
    $$D^{\frac{1}{2}} = \operatorname{diag}\left(\sqrt{\beta_{\deg(v_1)}}, \cdots,\sqrt{\beta_{\deg(v_n)}}\right),$$
    and the similarity $DA = D^{\frac{1}{2}}(D^{\frac{1}{2}}AD^{\frac{1}{2}})D^{-\frac{1}{2}}$, we have
    $$\spec(d_0f) = \spec(D^{\frac{1}{2}}(D^{\frac{1}{2}} A D^{\frac{1}{2}} + \alpha Id)D^{-\frac{1}{2}}) \subset \R. $$
    As $d_0f$ is similar to a Hermitian matrix, we have that (\ref{eq:beta_real}) follows by Weyl’s eigenvalue inequality.
    
    Assume now that $\beta_{\deg(c)}\leq 0$ for some cell $c$. It is immediate from (\ref{eq: d_0f}) that the  Hermitian  and skew-Hermitian parts of $d_0f$ turn out to be, respectively, 
    $$H = \frac{1}{2}(DA + AD) + \alpha Id, \hspace{1cm} K = \frac{1}{2}(DA - AD).$$
    From the spectral property,  for all $\lambda\in \spec(d_0f)$, $\operatorname{Re}(\lambda)\in[\lambda_{\min}(H),\lambda_{\max}(H)]$. Applying Weyl’s eigenvalue inequality again, for all $\sigma\in\spec(H)$, 
    $$\alpha+\lambda_{\min}\Big(\frac{1}{2}(DA+AD)\Big)\leq\sigma\leq \alpha + \lambda_{\max}\Big(\frac{1}{2}(DA+AD)\Big).$$
    and (\ref{eq:beta_neg}) follows.
\end{proof}

\subsection{Effects of Cherry Structures}
\label{subsec:4.2}

A cherry in a tree is a pair of leaves adjacent to a common neighbor. This configuration is a typical structural feature of trees: Erd\H{o}s and R\'enyi~\cite{ER} showed that the probability that a random tree contains at least one cherry tends to~1 as the order of the tree increases. Let us also introduce a more general notion of a cherry:

\begin{defi}
    Let $m \geq 2$ be an integer. An \textit{$m$-cherry} $C$ on a tree graph $G$ is a collection of $m$ leaves adjacent to a common vertex $w$, the \textit{center} of $C$. We say that $C_1, \ldots, C_k$ form a \textit{pack of cherries} if each $C_i$ is an $m_i$-cherry on $G$ and the sets $C_i$ are pairwise disjoint.

\end{defi}

The above definition allows a pack of cherries to contain cherries sharing the same center. The only requirement is that no leaf belongs to more than one cherry simultaneously.

Given a pack of cherries $C_1, \ldots, C_k$ in a tree graph $G$, consider the permutation $\phi_{C_1,\ldots,C_k}$ whose nontrivial cycles correspond to the sets $C_i$, for $i \in \{1,\ldots,k\}$. Then $\phi_{C_1,\ldots,C_k}$ is an automorphism of $G$  and so it induces the synchrony subspace
$$
\Delta_{C_1,\ldots,C_k} = \Fix\!\left(\langle \phi_{C_1,\ldots,C_k} \rangle\right),
$$
namely, the polydiagonal subspace associated with the pack of cherries. We denote by $W_{C_1,\ldots,C_k}$ the 
%complementary transversal 
orthogonal subspace to $\Delta_{C_1,\ldots,C_k}$.

If \(f\) is an admissible vector field on the binary tree of order seven (Fig.~\ref{fig:binary}) with an equilibrium at the origin, Example~\ref{ex:binary_tree} shows that a 2-cherry gives rise to \(\alpha\) as an eigenvalue of \(d_0 f\). This follows from the fact that the vertices \(v_4\) and \(v_5\), for instance, are both adjacent to \(v_2\), and any maximal matching of the graph can saturate at most one of these vertices, since a matching cannot contain two edges incident to the same vertex. As we show below, this phenomenon extends to a 2-cherry (Proposition~\ref{prop:eigenvalue}) and, more generally, to an \(m\)-cherry (Corollary~\ref{cor:m-cherry}) for arbitrary tree networks.

\begin{prop}
\label{prop:eigenvalue}
    Let $G$ be a tree network with a $2$-cherry $C =\{l_1,l_2\}$ and let $f$ be an admissible vector field on $G$. If $w$ is the center of $C$  and for  $\omega \in \Delta_{C}$, then the orthogonal subspace $W_{\{l_1,l_2\}}$ is invariant by $d_\omega f$ and $\partial_1h(\omega_{l_1},\omega_w)$ is an eigenvalue of $d_{\omega} f$ in the $W_{\{l_1,l_2\}}$-direction.
\end{prop}
\begin{proof}
   We assume, without loss of generality, that each cell phase space is one-dimensional. Let $\omega\in\Delta_{\{l_1,l_2\}}$. If $\mathbf{e}_i$ denotes the canonical  vector (the $i$-th entry equal to one and all entries other than $i$ equal to zero), then $W_{\{l_1,l_2\}} = \langle \mathbf{e}_{l_1} - \mathbf{e}_{l_2}\rangle$ and, by the admissibility of $f$, for $d_{\omega}f= (a_{ij})$ and  $i = l_1, l_2$, it follows that $a_{ii} = \frac{\partial f_i}{\partial x_i}(\omega) =\partial_1h(\omega_{l_i},\omega_w)$ and $a_{wi} = \frac{\partial f_w}{\partial x_i}(\omega) = \frac{\partial f_w}{\partial x_{l_1}}(\omega)$ are the only possible non-zero entries of $d_{\omega}f$ in the column $i$. Therefore, $\mathbf{e}_{l_1}-\mathbf{e}_{l_2}$ is an eigenvector of $d_\omega f$ associated with the eigenvalue $\partial_1h(\omega_{l_1},\omega_w)$,  which yields the result.
\end{proof} 

This is now generalized as follows:
  
\begin{cor} \label{cor:m-cherry}
    Let $G$ be a tree network with  an $m$-cherry $C = \{l_1,\cdots,l_m\}$ and let $f$ be an admissible vector field on $G$. If   $w$ is the center of $C$ and for $\omega\in\Delta_C$, then 

    \begin{itemize}
    \setlength{\itemindent}{0.5cm}
        \item $W_C=\bigoplus_{i=2}^m W_{\{l_1,l_i\}}$, 
        \item $W_{\{l_1,l_i\}}$ is invariant under $d_\omega f$, for $i = 2,\cdots, m$, 
        \item  $\partial_1h(\omega_{l_1},\omega_w)$ is an eigenvalue of $d_\omega f$ with algebraic multiplicity at least $m-1$.
    \end{itemize}
\end{cor}

\begin{proof}
   First, notice that $\omega \in \bigcap_{i=1}^m \Delta_{\{l_1,l_i\}} (= \Delta_C)$. By Proposition~\ref{prop:eigenvalue}, for  $i = 2,\cdots, m$, $W_{\{l_1,l_i\}}$ is invariant under $d_\omega f$ and $\partial_1h(\omega_{l_1},\omega_w)$ is an eigenvalue of $d_\omega f$ in $W_{\{l_1,l_i\}}$-direction. Since $\left\{W_{\{l_1,l_i\}}\right\}_{i\in\{2,\cdots,m\}}$ is a family of $m-1$ subspaces whose pairwise intersections are trivial, it follows that $\partial_1h(\omega_{l_1},\omega_w)$ is an eigenvalue of $d_\omega f$ with multiplicity at least $m-1$ and
   $$W_C = \Big(\bigcap_{i=1}^m \Delta_{\{l_1,l_i\}}\Big)^\perp = \sum_{i=2}^m\left(\Delta_{\{l_1,l_i\}}\right)^\perp = \bigoplus_{i=2}^m W_{\{l_1,l_i\}}.$$ 
\end{proof}

\subsection{Lyapunov Stability of Cherry Synchrony}
\label{subsec:4.3}

We start this subsection with an example:

\begin{ex}
\label{ex:nonlinear}
Consider again the network of Fig.~\ref{fig:binary}. We choose the following admissible vector field $f~:~\R^7\to\R^7$ whose components are given by
    \begin{align*}
        f_1(x_1,\cdots,x_7) &= x_1(x_2+x_3)^2\\
        f_2(x_1,\cdots,x_7) &=  x_2\tanh(x_1^2+x_4^2+x_5^2)\\
        f_3(x_1,\cdots,x_7) &= x_3\tanh(x_1^2+x_6^2+x_7^2)\\
        f_4(x_1,\cdots,x_7) &= -x_4(1+e^{-x_2^2})\\
        f_5(x_1,\cdots,x_7) &=  -x_5(1+e^{-x_2^2})\\
        f_6(x_1,\cdots,x_7) &= -x_6(1+e^{-x_3^2})\\
        f_7(x_1,\cdots,x_7) &= -x_7(1+e^{-x_3^2}).
    \end{align*}
Let $\pi$ be the orthogonal projection of $x \in \R^7\setminus \Delta_{\{v_4,v_5\}} $ on $W_{\{v_4,v_5\}}$. Define $$\mathcal{V}(x) = \|\pi(x)\|^2 = \frac{(x_{4}-x_{5})^2}{2}. $$
The Lie derivative of $\mathcal{V}$ along $f$ is given by
$$\dot{\mathcal{V}}(x) = \nabla \mathcal{V}(x)\cdot f(x) = (x_{4}-x_{5})(f_4(x) - f_5(x)) = (x_{4}-x_{5})(h(x_4,x_2) - h(x_5,x_2)).$$
Applying the Mean Value Theorem, given $x\in\R^7$, there exists  $\xi$ between $x_{4}$ and $x_{5}$ such that
$$h(x_4,x_2) - h(x_5,x_2)= \partial_1h(\xi,x_2)(x_{4}-x_{5}).$$
For $(a,b)\in\R^2$, $\partial_1h(a,b)= -1 - e^{-b^2} < -1$, so for any $x\in\R^7$,
$$\dot{\mathcal{V}}(x) = \partial_1h(\xi,x_2)(x_{4}-x_{5})^2 < (-1)(x_4-x_5) = -2\frac{(x_4-x_5)^2}{2} = -2\mathcal{V}(x).$$
Since $\mathcal{V}(x) = 0$ if and only if $x\in\Delta_{\{v_4,v_5\}}$, this implies that $\Delta_{\{v_4,v_5\}}$ is globally exponentially attracting in the sense of Lyapunov.
  
\end{ex}

\begin{rem}
    Let $\Delta_{C_1,\cdots,C_k}$ be the synchrony subspace associated with a pack of cherries $C_1,\cdots,C_k$, $U$ a tubular neighborhood of the synchrony subspace and $f$ an admissible vector field. The condition $x\in U$ restricts to the differences between the leaves coordinates, with no restriction on the relation between a leaf coordinate and a center vertex coordinate. Since the coupling function $h$ is evaluated only at one leaf coordinate and the center vertex, it follows that every pair of coordinates arises as the arguments of $h$ for some $x\in U$. Therefore, requiring $\partial_1h(a,b)$ to be negative on the tubular neighborhood $U$ is equivalent to  $\partial_1h(a,b)$ to be negative for all $(a,b)\in\mathbb R^2$, and so the condition is necessarily global.
 \end{rem}

The preceding example extends to Theorem~\ref{thm:Lyapunov}, whose proof relies on the same arguments. These arguments do not depend on the choice of coordinates, but only on the underlying cherry structure. For the precise formulation, we adopt the following convention: the inequality involving \(\partial_1 h\) is to be interpreted componentwise with respect to the variables of the first cell.

\begin{theorem}
\label{thm:Lyapunov}
    Let $G$ be a tree network, $C= \{l_1,l_2\}$ be a cherry and $f$ an admissible vector field on $G$. If $\partial_1h < N$  for some negative constant $N$ then $\Delta_{\{l_1,l_2\}}$ is Lyapunov globally exponentially attracting. 
\end{theorem}

The above Lyapunov stability condition for  2-cherry synchrony extends to a pack of cherries:

\begin{cor} 
\label{cor:final}
   Let $G$ be a tree graph, $C_1,\cdots,C_k$ a pack of cherries and $f$ an admissible vector field. If $\partial_1h < N$  for some negative constant $N$ then $\Delta_{C_1,\cdots,C_k}$ is Lyapunov globally exponentially attracting. 
\end{cor}
     
\begin{proof}
    For each $i\in\{1,\cdots,m\}$, let $l_i \in C_i$. The result follows directly from Theorem~\ref{thm:Lyapunov} and the fact that
    $$
    \Delta_{C_1,\cdots,C_k} = \bigcap_{i=1}^n \Delta_{C_i} = \bigcap_{i=1}^m\bigcap_{j\in C_i\setminus\{l_i\}} \Delta_{\{l_i,j\}}.
    $$
    \mbox{}
\end{proof}

\begin{acknowledgement}
This research has been financed by the S\~ao Paulo Research Foundation FAPESP, Brazil. Grant numbers
2024/08713-0 (N. Brito) and 2019/21181-0 (M. Manoel).
\end{acknowledgement}

%%%%%%%%%%%%%%%%%%%%%%%% referenc.tex %%%%%%%%%%%%%%%%%%%%%%%%%%%%%%
% sample references
% %
% Use this file as a template for your own input.
%
%%%%%%%%%%%%%%%%%%%%%%%% Springer-Verlag %%%%%%%%%%%%%%%%%%%%%%%%%%
%
% BibTeX users please use
% \bibliographystyle{}
% \bibliography{}
%

\end{document}